\documentclass[11pt]{amsart}
\usepackage{amsfonts, amsmath, amssymb, color}
\usepackage{amsthm}
\usepackage{booktabs}
\usepackage{float}
\usepackage{hhline}
\usepackage{lipsum}
\usepackage{lscape}
\usepackage{subfig}
\usepackage{textcomp}
\usepackage{tikz}
\usetikzlibrary{decorations.pathmorphing,shapes,arrows,positioning}
\usepackage{xcolor}
\usepackage{young}
\usepackage{youngtab}
\usepackage{ytableau}
\usepackage{hyperref}

\hypersetup{colorlinks,linkcolor=blue,urlcolor=cyan,citecolor=blue}
\allowdisplaybreaks[4]
\theoremstyle{plain}
\newtheorem{thm}{Theorem}[section]
\newtheorem{lem}[thm]{Lemma}
\newtheorem{prop}[thm]{Proposition}
\newtheorem{cor}[thm]{Corollary}
\newtheorem{rem}[thm]{Remark}

\theoremstyle{definition}

\newtheorem{exmp}[thm]{Example}

\newcommand{\la}{\lambda}

\newcommand{\tabincell}[2]{\begin{tabular}{@{}#1@{}}#2\end{tabular}}

\numberwithin{equation}{section} \errorcontextlines=0

\begin{document}
\title{Irreducible Characters and Bitrace for the $q$-Rook Monoid}
\author{Naihuan Jing}
\address{Department of Mathematics, North Carolina State University, Raleigh, NC 27695, USA}
\email{jing@ncsu.edu}
\author{Yu Wu}
\address{School of Mathematics, South China University of Technology,
Guangzhou, Guangdong 510640, China}
\email{wywymath@163.com}
\author{Ning Liu}
\address{Beijing International Center for Mathematical Research, Peking University, Beijing 100871, China}
\email{mathliu123@outlook.com}

\begin{abstract}
This paper studies irreducible characters of the $q$-rook monoid
algebra $R_n(q)$ using the vertex algebraic method. Based on the Frobenius formula for $R_n(q)$,
a new iterative character formula is derived with the help of the vertex operator realization of the Schur symmetric function.
The same idea also leads to a simple proof of the Murnaghan-Nakayama rule for $R_n(q)$. We also introduce the bitrace for the $q$-rook monoid and derive its combinatorial formula as a generalization of the bitrace formula 
for the Iwahori-Hecke algebra. The character table of $R_n(q)$ with $|\mu|=5$ is listed in the appendix.
\end{abstract}
\maketitle

\tableofcontents

\section{Introduction}
The rook monoid $R_n$ or symmetric inverse monoid is the monoid of partial permutations, i.e. $n\times n$ permutation matrices with some nonzero entries changed to $0$. Solomon \cite{S1} introduced a $q$-analogue $R_n(q)$ of the group algebra of $R_n$ as an ``Iwahori-Hecke algebra" for the monoid of $n\times n$ matrices $M_n(\mathbb{F}_q)$ over the finite field $\mathbb{F}_q$ with $q$ elements. Its ``Borel subgroup" of invertible upper triangular matrices, i.e., the $q$-rook monoid $R_n (q)$ is a semisimple $\mathbb{C}(q)$-algebra that specializes to $\mathbb{C}[R_n]$ when $q\rightarrow1$. In \cite{S2}, Solomon also obtained a presentation of $R_n(q)$ by generators and relations and defined an action of $R_n(q)$ on $U^{\otimes n}$, where $U=L(\omega_1)\oplus L(0)$ is the ($r+1$)-dimensional representation of $U_{q}(gl_r)$. Here $L(\omega_1)$ is 
the $r$-dimensional fundamental representation  and $L(0)$ the trivial representation. Halverson \cite{H} gave a different presentation of $R_n(q)$ and showed that as an $(U_q(gl_r),R_n(q))$-bimodule
\begin{align*}
    U^{\otimes n}\cong\bigoplus_{\la\in\hat{R}_n}L(\la)\otimes R_n^{\la}(q)
\end{align*}
summed over all partitions $\la$ with no more than $n$ boxes and $r$ rows. Here $R_n^{\la}(q)$ are simple $R_n(q)$-modules and $L(\la)$ are simple $U_q(gl_r)$-modules. A third presentation provided by Halverson and Ram showed that $R_n(q)$ is a quotient of the affine Hecke algebra in type $A$.

In \cite{DHP}, Dieng, Halverson and Poladian developed the character theory of $R_n(q)$ as a generalization of the earlier work of Munn \cite{Mu}. They used a Schur-Weyl duality between $R_n(q)$ and the quantum general linear group $U_q(gl_r)$ to give a Frobenius formula, in the ring of symmetric functions, for the irreducible characters of $R_n(q)$. A recursive Murnaghan-Nakayama rule for these characters $\chi^{\la}_{\mu}(q)$ was also derived therein, which expresses $\chi^{\la}_{\mu}(q)$ in terms of those whose lower partitions have shorter length.

The aim of this paper is to find an effective and computation friendly formulae for $\chi^{\la}_{\mu}(q)$. We use vertex algebraic techniques to give a new iterative formula for $\chi^{\la}_{\mu}(q)$. Different from the Murnaghan-Nakayama rule, our new formula express $\chi^{\la}_{\mu}(q)$ in terms of those whose upper partitions have shorter length. Many special cases are considered. In particular, compact rules for $\chi^{\la}_{\mu}(q)$ with $\la$ hook and two-row type are derived respectively. 

According to \cite{HLR}, the bitrace can be viewed a deformed orthogonality relation for irreducible characters of the Hecke algebra. We also introduce the bitrace of the $q$-rook monoid algebra and derive a 
parallel combinatorial formula.

The paper is organized as follows. In the next section, we will review briefly some essential preliminaries about symmetric functions, vertex operator realization of Schur functions and some structures of the $q$-rook monoid algebra. In Sec. \ref{s:formula}, we will give our main results: a new recursive formula for $\chi^{\la}_{\mu}(q)$ (Theorem \ref{X,X}) and the closed formulae for the cases of $\la$ being hook and two-row type (Theorem \ref{t:compact}). The Murnaghan-Nakayama rule is also redrived. Sec. \ref{s:bitrace} deals with the bitrace of the $q$-rook monoid algebra. A general combinatorial formula for the bitrace is deduced (Theorem \ref{t:bitrace}). In Appendix \ref{app}, we list a table of $\chi^{\la}_{\mu}(q)$ with $|\mu|=5$ using our formulae.

\section{Preliminaries}
In this section, first we present some well-known facts about symmetric functions \cite{Mac}, which will be used
in the subsequent sections. Then we will review briefly the vertex operator realization of Schur functions and the Frobenius formula for the $q$-rook monoid algebra.

\subsection{Partitions and symmetric functions} A {\it partition} $\lambda$ of $n>0$ is a sequence of non-negative integers $(\lambda_1,\lambda_2,\ldots)$ such that $\lambda_1\geqslant\lambda_2\geqslant\ldots$ and $\sum_i\la_i=|\la|=n$. We will call $\emptyset$ a partition of $0$. The nonzero $\lambda_i$ are called {\it parts} of $\lambda$, and $n$ is called the {\it weight} of $\la$ and denoted by $\la\vdash n$.
The number of nonzero parts is {\it length} $l(\la)$.  Sometimes $\lambda$ is arranged in the ascending order: $\lambda=(1^{m_1}2^{m_2}\ldots)$ with $m_i$ being the multiplicity of part $i$ in $\lambda$. The conjugate $\la'$ of a partition $\la$ is the partition defined by $\la^\prime_i={\rm Card}\{{j:\la_j\geqslant i}\}$ and $m_i=\la^\prime_i-\la^\prime_{i+1}.$

A {\it composition} $\tau$ of $n$ is a sequence of non-negative integers $(\tau_1,\ldots,\tau_l)$ satisfying $\sum_i{\tau_i}=n$, denoted by $\tau\models n$. Clearly, a partition is a composition. For any two compositions $\tau$ and $\rho$, we denote $\tau\subset\rho$ if $\tau_i\leqslant\rho_i$ for all $i\geqslant1$. We define the operations ``+" and ``-" for compositions (or partitions) by $\tau\pm\rho:=(\tau_1\pm\rho_1,\tau_2\pm\rho_2,\ldots)$. Here we add 0's at the tail of $\tau$ or $\rho$ such that they have the same number of components.

A partition $\la$ can be visualized by its Young diagram, which consists of
left adjusted array of boxes (also called nodes) with $\la_i$ boxes on the $i$th row.
For two partitions $\la\supset\mu$, the set-theoretic difference $\theta=\la/\mu$ is called a {\it skew diagram}. A skew diagram $\theta=\la/\mu$ is a {\it vertical} (resp. {\it horizontal} ) {\it $m$-strip} if $|\theta|=m$ and $\theta_i=\la_i-\mu_i\leqslant1$ (resp. $\theta_i^{'}=\la_i^{'}-\mu_i^{'}\leqslant1$) for each $i\geqslant1$. In other words, a vertical (resp. horizontal) strip has at most one square in each row (resp. column).

Let $\Lambda$ be the ring of symmetric functions in the $x_n$ ($n\in\mathbb N$) over $\mathbb{Q}$, and we will also study the ring $\Lambda_{\mathbb{Z}}$ as a lattice of $\Lambda$. The ring $\Lambda$ has several linear bases indexed by partitions. For each $r>1$, let $p_{r}=\sum x_{i}^{r}$ be the $r$th power-sum. Then $p_{\lambda}=p_{\lambda_{1}}p_{\lambda_{2}}\cdots p_{\lambda_{l}}, \lambda\in\mathcal P$ form a $\mathbb Q$-basis of $\Lambda$.
We introduce the modified power sum functions $\hat{p}_n=p_n(1,x_1,x_2,\ldots)=1+p_n$.

The ring $\Lambda$ also has various $\mathbb{Z}$-bases indexed by partitions: the set of monomial symmetric functions $m_{\la}=\sum x_{i_1}^{\la_1}\cdots x_{i_k}^{\la_k}$; that of elementary symmetric functions $e_{\la}=e_{\la_1}\cdots e_{\la_k}$ with $e_n=m_{(1^n)}$; that of complete symmetric function $h_{\la}=h_{\la_1}\cdots h_{\la_k}$ with $h_n=\sum_{\la\vdash n}m_{\la}$, and that of Schur symmetric functions $s_{\la}$.

The space $\Lambda$ is equipped with the Hall inner product defined by \cite[\S 1.4]{Mac})
\begin{align}\label{e:inner}
\langle p_{\lambda}, p_{\mu}\rangle=\delta_{\lambda\mu}z_{\lambda},
\end{align}
where $z_{\lambda}=\prod_{i\geq 1}\lambda_im_i(\lambda)!$, under which the Schur functions $s_{\lambda}$'s form an orthonormal basis of $\Lambda$.

The function $p_n$ acts on $\Lambda$ as a multiplication operator, and its adjoint operator is the differential operator $p_n^* =n\frac{\partial}{\partial p_n}$. Note that * is $\mathbb Q$-linear and  anti-involutive satisfying
\begin{equation}
\langle p_nu, v\rangle=\langle u, p_n^*v\rangle
\end{equation}
for $u, v\in \Lambda$.

Using the degree gradation, $\Lambda$ becomes a graded ring
\begin{align}
\Lambda=\bigoplus_{n=0}^{\infty} \Lambda_n.
\end{align}
For a given operator $A$ on $\Lambda$, we say $A$ is of degree $n$ if $A$ maps a polynomial of degree $m$ to a polynomial of degree $m+n$, i.e., $A(\Lambda_m)\subset \Lambda_{m+n}$. Thus, $p_n$ (resp. $p_n^*$) is of degree $n$ (resp. $-n$) as the multiplication (resp. differential) operator on $\Lambda$.
\subsection{Vertex operator realization of Schur functions}
We recall the vertex operator realization of Schur functions. Let $S(z)$ be the map: $\Lambda\longrightarrow \Lambda[[z, z^{-1}]]$ defined by (see \cite{Jing2})
\begin{align}\label{e:Schurop}
S(z)&=\mbox{exp} \left( \sum\limits_{n\geq 1} \dfrac{1}{n}p_nz^{n} \right) \mbox{exp} \left( -\sum \limits_{n\geq 1} \frac{\partial}{\partial p_n}z^{-n} \right)=\sum_{n\in\mathbb Z}S_nz^{n}.
\end{align}

Note that $p_n^* =n\frac{\partial}{\partial p_n}$, we have the dual or adjoint operator of $S(z)$:
\begin{align}
S^*(z)&=\exp\left(-\sum_{n\geqslant1}^\infty \frac{1}{n}p_nz^n\right)\exp\left(\sum_{n\geqslant1}^\infty \dfrac{\partial}{\partial p_n}z^{-n}\right)=\sum_{n\in\mathbb{Z}}S_n^*z^{-n}.
\end{align}

We use the convention to index the components by their degrees.
The operators $S_n \in \mathrm{End}(\Lambda)$ are the Bernstein vertex operators realizing the Schur functions.

\begin{prop}\label{p:vo}\cite{Jing1,Jing2}
(1)  The components of $S(z)$ and $S^*(z)$ obey the following commutation relations:
    \begin{align}
        S_mS_n+S_{n-1}S_{m+1}&=0,\\
        S_m^*S_n^*+S_{n+1}^*S_{m-1}^*&=0,\\
        S_mS_n^*+S_{n-1}^*S_{m-1}&=\delta_{m,n}.
    \end{align}

(2) For any composition $\mu=(\mu_{1},\ldots,\mu_{k})$, the product $S_{\mu_{1}}\cdots S_{\mu_{k}}.1=s_{\mu}$ is the Schur function labeled by $\mu$. In general, $s_{\mu}=0$ or $\pm s_{\lambda}$ for a partition $\lambda$ such that $\la=\sigma(\mu+\delta)-\delta$ for some $\sigma\in\mathfrak S_k$. Here $\delta=(k-1,k-2,\ldots,1,0)$. Moreover, $S_{-n}.1=\delta_{n,0}, S^{*}_{n}.1=\delta_{n,0}, (n\geq0)$.
\end{prop}

\begin{exmp}
    Here we provide an explicit example to illustrate Proposition \ref{p:vo} (2). We choose $\mu=(2,1)$. Then 
    \begin{align*}
        &S_{(2,1)}.1=S_2S_1.1=S_2p_1\\
        =&\left(\frac{1}{2}p_2+\frac{1}{2}p_1p_1+\left(\frac{1}{3}p_3+\frac{1}{2}p_2p_1+\frac{1}{6}p_1^3\right)\left(-\frac{\partial}{\partial p_1}\right)+\text{\small terms with differential operators of degree $>1$ }\right)p_1.
\end{align*}
Note that $p_1$ will be killed by the terms with differential operators of degree $>1$. Thus,
\begin{align*}
         S_{(2,1)}.1=&\frac{1}{2}p_2p_1+\frac{1}{2}p_1^3-\frac{1}{3}p_3-\frac{1}{2}p_2p_1-\frac{1}{6}p_1^3\\
        =&\frac{1}{3}p_1^3-\frac{1}{3}p_3\\
        =&s_{(2,1)}.
    \end{align*}
    
\end{exmp}

The special Schur function $h_n=S_n.1$ is called the complete homogeneous symmetric function, whose generating function is given by
\begin{align}
\exp\left( \sum\limits_{n\geq 1} \frac{1}{n}p_nz^{n} \right)=\prod_{i\geq1}\frac{1}{1-x_i z}=\sum_{n\geq0}h_nz^n.
\end{align}
Similarly, its dual operator can be written as
\begin{align}
\exp\left( \sum\limits_{n\geq 1} \frac{\partial}{\partial p_n}z^{-n} \right)=\sum_{n\geq0}h^*_nz^{-n}.
\end{align}

Introduce the symmetric functions $q_n(t)$ defined by its generating function:
\begin{align}\label{e:qfcn1}
\exp\left(\sum_{n=1}^{\infty}\frac{1-t^n}np_nz^n\right)=\sum_{n=0}^{\infty}q_n(t)z^n=q(z;t)
\end{align}
where $q_n(t)$ is the Hall-Littlewood polynomial of one-row partition $(n)$, and clearly
\begin{align}\label{e:qfcn2}
q_n(t)=\sum_{\lambda\vdash n}\frac1{z_{\lambda}(t)}p_{\lambda}.
\end{align}
Here $z_{\lambda}(t)=\frac{z_{\lambda}}{\prod_{i\geq 1}(1-t^{\lambda_i})}$.
For each partition $\lambda$, denote $q_{\lambda}(t)=q_{\lambda_1}(t)q_{\lambda_2}(t)\cdots$, then
$\{q_{\lambda}(t)\}$ forms a basis of $\Lambda\otimes\mathbb{Q}(t)$.

Recall that $\hat{p}_n=p_n(1,x_1,x_2,\cdots)$. We introduce a modification $\hat{q}_n(t)$ of $q_n(t)$ as follows:
\begin{align}
\exp\left(\sum_{n=1}^{\infty}\frac{1-t^n}n\hat{p}_nz^n\right)=\sum_{n=0}^{\infty}\hat{q}_n(t)z^n=\hat{q}(z;t).
\end{align}
Similarly, we have $\hat{q}_n(t)=\sum_{\lambda\vdash n}\frac1{z_{\lambda}(t)}\hat{p}_{\lambda}.$ For a partition $\lambda$, denote $\hat{q}_{\lambda}(t)=\hat{q}_{\lambda_1}(t)\hat{q}_{\lambda_2}(t)\cdots$.

\begin{lem}\label{t:relq^}
The relation between $\hat{q}_{\la}(t)$ and $q_{\la}(t)$ is given by
\begin{align}\label{e:relq^}
\hat{q}_{\la}(t)=\sum_{\tau}(1-t)^{l(\tau)}q_{\la-\tau}(t)
\end{align}
summed over all compositions $\tau$ such that $\tau\subset\la$. Here $\la-\tau=(\la_1-\tau_1,\la_2-\tau_2,\cdots)$. In particular, $\hat{q}_n(t)=q_n(t)+(1-t)\sum_{i\geqslant1}^n q_{n-i}(t).$
\end{lem}
\begin{proof}
By definition it follows that
\begin{align*}
\sum_{n=0}^{\infty}\hat{q}_n(t)z^n&=\exp\left(\sum_{n=1}^{\infty}\frac{1-t^n}n\hat{p}_nz^n\right)\\
&=\exp\left(\sum_{n=1}^{\infty}\frac{1-t^n}np_nz^n\right)\exp\left(\sum_{n=1}^{\infty}\frac{1-t^n}nz^n\right)\\
&=\left(\sum_{n=0}^{\infty}q_n(t)z^n\right)\left(1+(1-t)\sum_{i\geqslant1}^\infty z^i\right).
\end{align*}
Comparing the coefficient of $z^n$ on both sides gives $\hat{q}_n(t)=q_n(t)+(1-t)\sum_{i\geqslant1}^n q_{n-i}(t)$, which yields \eqref{e:relq^} by induction.
\end{proof}

\subsection{Irreducible characters for the $q$-rook monoid algebra}

Let $q$ be an indeterminate. For $n\geq2$, the $q$-rook monoid $R_n(q)$ is defined as the associative $\mathbb{C}(q)$-algebra with generators $1, T_1,\cdots, T_{n-1}$,\\$ P_1, \cdots , P_n$ subject to the relations \cite{H}:
\begin{align*}
(T_i-q)(T_i+1)=0, \quad &1\leq i\leq n-1,\\
T_iT_j=T_jT_i, \quad &1\leq i,j\leq n-1, \mid i-j\mid>1,\\
T_iT_{i+1}T_i=T_{i+1}T_iT_{i+1}, \quad &1\leq i\leq n-2,\\
T_iP_j=P_jT_i=qP_j, \quad &1\leq i< j\leq n,\\
T_iP_j=P_jT_i, \quad &1\leq j< i\leq n-1,\\
P^2_{i}=P_i, \quad &1\leq i\leq n,\\
P_{i+1}=qP_iT^{-1}_{i}P_i, \quad &2\leq i \leq n.
\end{align*}

Define $R_0(q) = \mathbb{C}(q)$, and $R_1(q)$ is the associative $\mathbb{C}(q)$-algebra spanned by $1$ and
$P_1$ subject to $P^2_1 = P_1$. The subalgebra of $R_n(q)$ generated by $T_1,\cdots, T_{n-1}$ is isomorphic
to the Iwahori-Hecke algebra $H_n(q)$ in type $A$. Here we adopt Halverson's presentation of $R_{n}(q)$ \cite{H} for our purposes.

Let $s_i\in\mathfrak S_n$ be the transposition of $i$ and $i+1$, we define $\gamma_1=T_{\gamma_1}=1$ and for $2\leq t\leq n$,
\begin{align*}
    \gamma_t&=s_1s_2\cdots s_{t-1},\\
    T_{\gamma_t}&=T_1\cdots T_{t-1}.
\end{align*}
For a composition $\mu=(\mu_1,\ldots,\mu_l)\models k,\;0\leq k\leq n$, we define
\begin{align*}
    T_{\gamma_\mu}&=T_{\gamma_{\mu_1}}\otimes\ldots\otimes T_{\gamma_{\mu_l}},\\
    T_{\mu}&=P_{n-k}\otimes T_{\gamma_{\mu}},
\end{align*}
where the tensor product means that the factors $T_{\gamma_{\mu_i}}$ are generated by $T_j$ with $j=\sum_{a=1}^{j-1} \mu_a+1, \sum_{a=1}^{j-1} \mu_a+2,
\ldots, \sum_{a=1}^{j} \mu_a$ (disjoint indices).
Here $T_\mu\in R_{n-k}(q)\otimes R_{\mu_1}(q)\otimes\ldots\otimes R_{\mu_l}(q)\subset R_n(q)$ is called the {\it standard element}. It is clear that $T_{\mu}$ is the identity when $\mu=(1^n)$.

If $K\subset\{1,2,\ldots,n\}$ define the subgroup $\mathfrak S_K\subset\mathfrak S_n$ as the group of permutations on the elements of $K$. For $1\leq i\leq n$,  we define $T_{i,i}=1$, and define for $1\leq i<j\leq n$
\begin{align}
    T_{i,j}=T_{j-1}T_{j-2}\ldots T_i.
\end{align}
Let $A=\{a_1,a_2,\ldots,a_k\}\subset \{1,2,\ldots,n\}$ and assume that $a_1<a_2<\ldots<a_k$. Define
\begin{align}
    T_A=T_{1,a_1}T_{2,a_2}\ldots T_{k,a_k}.
\end{align}

For $0\leq k\leq n$, let $\Omega_k$ be the following set of triples,
\begin{align*}
    \Omega_k=\left\{(A,B,\omega)\Bigg|\begin{array}{lc}A,B\subset\{1,2,\ldots,n\}, |A|=|B|=k,\\\omega\in\mathfrak S_{\{k+1,\ldots,n\}}\end{array}\right\}
\end{align*}
and let $\Omega=\bigcup_{k=0}^n\Omega_k$. If we define {\it the standard words}
\begin{align}
T_{(A,B,\omega)}=T_AT_\omega P_kT_B^{-1}, (A,B,\omega)\in\Omega_k.
\end{align}
Then the set $\{T_{(A,B,\omega)}|(A,B,\omega\in\Omega)\}$ forms a linear $\mathbb{C}(q)$-basis of $R_n(q)$ \cite{H}.

It is known that irreducible characters of $R_n(q)$ are parameterized by partitions of $k\leq n, k=0, 1, \ldots, n$. Here a partition of $0$ means $\emptyset$.
The irreducible module $V^{\la}$ indexed by $\la\vdash k\leq n$ have been studied
in \cite{G, Mu, S1b}, and the seminormal bases of $V^{\la}$ were found in \cite{H} (which will be recalled in the next section). 
Dieng, Halverson and Poladian \cite{DHP} used the Schur-Weyl
duality to prove the following Frobenius-type formula in the ring of symmetric functions ($\mu\vdash n$)
\begin{align}\label{chara formula}
\frac{q^{|\mu|}}{(q-1)^{l(\mu)}}\hat{q}_{\mu}(q^{-1})=\sum_{k=0}^{n}\sum_{\la\vdash k}\chi^{\la}_{R_n(q)}(T_{\mu})s_{\la}
\end{align}
where $\chi^{\la}_{R_n(q)}(T_{\mu})$ is the value of the irreducible characters of $R_n(q)$ indexed by $\la\vdash k (\leq n)$ at the standard element $T_{\mu}$.
In the following we sometimes write $\chi^{\la}_{\mu}(q)$ for $\chi^{\la}_{R_n(q)}(T_{\mu})$ for brevity.
Formula \eqref{chara formula} is a generalization of the Frobenius formula given by Ram \cite{Ram} for the Iwahori-Hecke algebra $H_n(q)$ in type $A$. The irreducible
$R_n(q)$-characters are completely determined by their values on the set of standard elements $T_{\mu}, \mu\vdash k, 0 \leq k \leq n.$ Specifically, the character table of $R_n(q)$ is block upper triangular with the diagonal blocks being the character tables of the Iwahori-Hecke algebras $H_k(q)$ in type $A$, $k\leq n$.

By the  orthonormality and the vertex operator realization of $s_{\la}$,
\begin{align*}
\chi^{\la}_\mu(q)=\frac{q^{|\mu|}}{(q-1)^{l(\mu)}}\langle \hat{q}_\mu(q^{-1}), s_{\la}\rangle=\frac{q^{|\mu|}}{(q-1)^{l(\mu)}}\langle \hat{q}_\mu(q^{-1}), S_{\la}.1\rangle.
\end{align*}
Denote $X_\mu^\lambda(t)=\langle \hat{q}_\mu(t), S_\lambda.1\rangle$, then $\chi^{\la}_\mu(q)=\frac{q^{|\mu|}}{(q-1)^{l(\mu)}}X_\mu^\lambda(q^{-1})$.
We are going to compute $X_\mu^\lambda(t)$ in the sequel.

\section{Irreducible characters}\label{s:formula}
In this section, we will give a new iterative formula, which implies compact formulae for some special cases. Moreover, we rederive the Murnaghan-Nakayama rule based on the vertex operator realization of Schur functions and the Frobenius formula.

It is known that the Young seminormal representation \cite{Hoe} of the Hecke algebra $H_n(q)$ can be extended to
that of $R_n(q)$ \cite{H}. For $\la\vdash k$ ($k\leq n$), an $n$-standard tableau $L$ of shape $\la$ is a filling of the diagram of $\la$ with numbers from $\{1,2,\ldots,n\}$ such that (1) each number appears at most once, (2) the entries in each column strictly increase from top to bottom, and (3) the entries in each row strictly increase from left to right.

If $b$ is a box of $L$ in position $(i,j)\in\la$, the {\it content} of $b$ is defined as
\begin{align}
c(b)=j-i.
\end{align}
For $i\in\{1, \ldots, n\}$, we sometimes call $i\in L$ if $i$ appears inside $L$, and refer $L(i)$ as the box of $L$ where $i$ stands.

The number $f_{\la}$ of $n$-standard tableaux of shape $\la\vdash n$ is given by
$f_{\la}={n!}/\prod_{(i,j)\in\la}h_{ij}$, where $h_{ij}=\la_i+\la_j'-i-j+1$ is called the {\em hook-length} of $\la$ at $(i,j)\in\la$. Subsequently
the number of $n$-standard tableaux of shape $\la\vdash k$ ($k\leq n$) is given by $\binom{n}{k}f_{\la}$.

Let $T_n^\lambda$ be the set of $n$-standard tableaux of shape $\la\vdash k$ ($k\leq n$), and define
$V^{\la}$ as the $\mathbb{C}(q^{1/2})$-vector space spanned by $v_L, L\in T_n^{\lambda}$. We
define an action of
$R_n(q)$ on $V^{\lambda}$ as follows: $P_iv_L=v_i$ for $i\notin L$ otherwise $P_iv_L=0$ and
\begin{align}
T_iv_L=\begin{cases}
\dfrac{q-1}{1-q^{c(L(i))-c(L(i+1))}}v_L+(1+\dfrac{q-1}{1-q^{c(L(i))-c(L(i+1))}})v_{L^\prime},&\mbox{if $i,i+1\in L$},\\
(q-1)v_L+q^{1/2}v_{s_iL},\quad&\mbox{if $i\notin L$, $i+1\in L$},\\
q^{1/2}v_{s_iL},\quad\quad&\mbox{if $i\in L$, $i+1\notin L$},\\
qv_L,\quad\quad&\mbox{if $i,i+1\notin L$},
\end{cases}
\end{align}
where
\begin{align}
v_{L^\prime}=\begin{cases}v_{s_iL},&\mbox{if $s_iL$ is $n$-standard},\\
                          0,\quad&\mbox{otherwise}.
\end{cases}
\end{align}


\begin{thm}\cite{H}
For each $\la\vdash k\leq n$, the above $R_n(q)$-action on $V^\lambda$ affords an irreducible representation of $R_n(q)$. Moveover, the set $V^\lambda, \la\vdash k\leq n$, is a complete set of irreducible, pairwise non-isomorphic $R_n(q)-$modules.
\end{thm}
In the following we will simply use $\chi^{\la}$ to denote the irreducible character of $V^{\la}$.

\subsection{A new iterative formula}For each partition $\la=(\la_1, \ldots, \la_l)$, we define that
\begin{align}
\la^{[i]}=(\la_{i+1}, \cdots, \la_l), \qquad i=0, 1, \ldots, l
\end{align}
So $\la^{[0]}=\la$ and $\la^{[l]}=\emptyset$.

In \cite[Sec. 5]{JL2}, two of us obtained the following decomposition of Schur functions. For an arbitrary partition $\la=(\la_1,\la_2,\cdots,\la_l)\vdash n$, then
\begin{align}\label{e:decomposition}
s_{\la}=\sum_{\nu}(-1)^{n-|\nu|-\la_1}h_{n-|\nu|}s_{\nu}
\end{align}
summed over all partitions $\nu\subset\la^{[1]}$ such that $\la^{[1]}/\nu$ is a vertical strip.

\begin{prop}\label{t:h*q}
Let $k$ be a nonnegative integer and $\mu=(\mu_1,\mu_2,\cdots,\mu_m)\vdash n$ be a partition, then we have
\begin{align}\label{e:h*q}
h^{*}_{k}\hat{q}_{\mu}.1=\sum_{\tau\in C(\mu;k)}(1-t)^{l(\tau)}\hat{q}_{\mu-\tau}.1,
\end{align}
where $C(\mu;k):=\{\tau\subset\mu\mid \tau\models k\}$.
\end{prop}
\begin{proof} Note that $[\frac{\partial}{\partial p_{n}}, \hat{p}_m]=\delta_{n,m}$. The usual vertex operator calculus gives that
\begin{align*}
\mbox{exp} \left( \sum\limits_{n=1}^{\infty}\frac{\partial}{\partial p_{n}}z^{-n} \right)\mbox{exp} \left( \sum\limits_{n=1}^{\infty}\frac{1-t^{n}}{n}\hat{p}_{n}w^{n} \right)=\mbox{exp} \left( \sum\limits_{n=1}^{\infty}\frac{1-t^{n}}{n}\hat{p}_{n}w^{n} \right)\mbox{exp} \left( \sum\limits_{n=1}^{\infty}\frac{\partial}{\partial p_{n}}z^{-n} \right)\frac{z-tw}{z-w}.
\end{align*}
Using repeatedly the above commutative relation noting that $\frac{\partial}{\partial p_n}.1=0$, we have
\begin{align}\label{e:commutative}
\begin{split}
&\mbox{exp} \left( \sum\limits_{n=1}^{\infty}\frac{\partial}{\partial p_{n}}z^{-n} \right)\mbox{exp} \left( \sum\limits_{n=1}^{\infty}\frac{1-t^{n}}{n}\hat{p}_{n}w_1^{n} \right)\cdots\mbox{exp} \left( \sum\limits_{n=1}^{\infty}\frac{1-t^{n}}{n}\hat{p}_{n}w_m^{n} \right).1\\
=&\mbox{exp} \left( \sum\limits_{n=1}^{\infty}\frac{1-t^{n}}{n}\hat{p}_{n}w_1^{n} \right)\cdots\mbox{exp} \left( \sum\limits_{n=1}^{\infty}\frac{1-t^{n}}{n}\hat{p}_{n}w_m^{n} \right)\mbox{exp}\prod_{m\geq i\geq1}\frac{z-tw_i}{z-w_i}.
\end{split}
\end{align}

The coefficient of $z^{-k}w_1^{\mu_1}w_2^{\mu_2}\cdots w_{m}^{\mu_m}$ is $h^*_{k}q_{\mu}$ in \eqref{e:commutative} equals $h^*_{k}q_{\mu}$,
and
the right-hand side in \eqref{e:commutative} equals to
\begin{align*}
&\prod_{j=1}^{m}\left(1+(1-t)\sum_{l\geq1}(\frac{w_j}{z})^{l} \right)\prod_{j=1}^{m}\left(\sum_{n\geq0}\hat{q}_nw_{j}^{n}\right)\left(\sum_{n\geq0}h^*_nz^{-n}\right).1\\
=&\prod_{j=1}^{m}\left(1+(1-t)\sum_{l\geq1}(\frac{w_j}{z})^{l} \right)\prod_{j=1}^{m}\left(\sum_{n\geq0}\hat{q}_nw_{j}^{n}\right).1,
\end{align*}
in which the corresponding coefficient of $z^{-k}w_1^{\mu_1}w_2^{\mu_2}\cdots w_{m}^{\mu_m}$ is exactly $\sum_{\tau\in C(\mu;k)}(1-t)^{l(\tau)}\hat{q}_{\mu-\tau}.1$
\end{proof}

Now we state our main result.
\begin{thm}\label{X,X}
Let $\mu$ be a partition of $n$, $\la=(\la_1,\ldots,\la_r)$ be a partition of $m\leq n$. We have that
\begin{align}\label{e:ite}
\chi_\mu^{\la}(q)=\sum_{\nu}\sum_{\tau\in C(\mu;m-|\nu|)}(-1)^{m-|\nu|-\la_1}\frac{q^{m-|\nu|-l(\tau)}}{(q-1)^{l(\mu)-l(\tau)-l(\mu-\tau)}}\chi_{\mu-\tau}^{\nu}(q)
\end{align}
summed over all partitions $\nu\subset\la^{[1]}$ such that $\la^{[1]}/\nu$ is a vertical strip.
\end{thm}
\begin{proof} Direct computation gives that 
\begin{align*}
X_\mu^\lambda(t)&=\left\langle\hat{q}_\mu.1, S_\lambda.1 \right\rangle\\
                &=\sum_{\nu}(-1)^{m-|\nu|-\la_1}\left\langle \hat{q}_\mu.1,h_{m-|\nu|}s_{\nu}\right\rangle \quad\text{(by \eqref{e:decomposition})}\\
                &=\sum_{\nu}(-1)^{m-|\nu|-\la_1}\left\langle h_{m-|\nu|}^*\hat{q}_\mu.1,s_{\nu}\right\rangle\\
                &=\sum_{\nu}(-1)^{m-|\nu|-\la_1}\left\langle \sum_{\tau\in C(\mu;m-|\nu|)}(1-t)^{l(\tau)}\hat{q}_{\mu-\tau}.1,s_{\nu}\right\rangle \quad \text{(by \eqref{e:h*q})}\\
                &=\sum_{\nu}\sum_{\tau\in C(\mu;m-|\nu|)}(-1)^{m-|\nu|-\la_1}(1-t)^{l(\tau)}X^{\nu}_{\mu-\tau}(t)
\end{align*}
summed over all partitions $\nu\subset\la^{[1]}$ such that $\la^{[1]}/\nu$ is a vertical strip. Then \eqref{e:ite} follows from $\chi^{\la}_\mu(q)=\frac{q^{|\mu|}}{(q-1)^{l(\mu)}}X_\mu^\lambda(q^{-1})$.
\end{proof}
We remark that in \eqref{e:ite} $\mu-\tau$ is seen as a partition obtained by arranging parts of composition $\mu-\tau$ in descending order.

\begin{exmp}
Suppose $\la=(31^2)$, $\mu=(321)$. The possible chooses of $\nu$ are $\emptyset, (1), (1,1)$. By Theorem \ref{X,X}, we have
\begin{align*}
\chi_{(321)}^{(31^2)}(q)=&\sum_{\tau\in C(\mu;5)}\dfrac{q^{5-l(\tau)}}{(q-1)^{3-l(\tau)-l(\mu-\tau)}}\chi_{\mu-\tau}^{\emptyset}(q)-\sum_{\tau\in C(\mu;4)}\dfrac{q^{4-l(\tau)}}{(q-1)^{3-l(\tau)-l(\mu-\tau)}}\chi_{\mu-\tau}^{(1)}(q)\\
&+\sum_{\tau\in C(\mu;3)}\dfrac{q^{3-l(\tau)}}{(q-1)^{3-l(\tau)-l(\mu-\tau)}}\chi_{\mu-\tau}^{(1^2)}(q)\\
=&(3q^3-2q^2)\chi^{\emptyset}_{(1)}(q)-(2q^2-q)\chi_{(2)}^{(1)}(q)-(3q^3-4q^2+q)\chi_{(1^2)}^{(1)}(q)\\
&+q\chi^{(1^2)}_{(3)}(q)+(4q^2-4q+1)\chi_{(21)}^{(1^2)}(q)+q(q-1)^2\chi^{(1^2)}_{(1^3)}(q).
\end{align*}
So $\chi_{(321)}^{(31^2)}(q)=(q-1)(2q^2-8q+2).$
\end{exmp}

\begin{rem} In \eqref{e:ite} there appears the character $\chi^{\emptyset}_{(1)}(q)$ which is
not necessarily equal to 1. In fact, we have that
\begin{align*}
\chi^{\emptyset}_{\mu}(q)&=\frac{q^{|\mu|}}{(q-1)^{l(\mu)}}\langle \hat{q}_{\mu}(q^{-1}), 1 \rangle\\
&=\frac{q^{|\mu|}}{(q-1)^{l(\mu)}}\times\text{constant term of $\hat{q}_{\mu}(q^{-1})$}\\
&=q^{|\mu|-l(\mu)} \quad \text{(by \eqref{e:relq^})}.
\end{align*}
\end{rem}

\subsection{Murnaghan-Nakayama rule}
A {\it path} in a skew diagram $\theta$ is a sequence $x_0,x_1,\cdots,x_n$ of squares in $\theta$ such that $x_{i-1}$ and $x_i$ have a common side, for $1\leq i\leq n.$ A subset $\xi$ of $\theta$ is connected if any two squares in $\xi$ are connected by a path in $\xi$. The connected components, themselves skew diagrams, are by definition the maximal connected subsets of $\theta$. A skew diagram $\theta$ is a {\it border strip} if it is connected and
contains no $2\times 2$ blocks of squares (cf. \cite[p. 5]{Mac}). A skew diagram is a {\it generalized border strip} (written as GBS in short) if it
is a union of connected border strips. Any GBS is a union of its connected components, each of which is a border strip. A $k$-GBS is a GBS with $k$ boxes. In the example below, $\lambda=(4,3,3,1),$ $\mu=(3,2,1),$ the GBS $\lambda/\mu$ has three border strips. It's a 5-GBS.
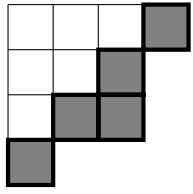
\begin{gather*}
  \centering
\begin{tikzpicture}[scale=0.6]
   \coordinate (Origin)   at (0,0);
    \coordinate (XAxisMin) at (0,0);
    \coordinate (XAxisMax) at (4,0);
    \coordinate (YAxisMin) at (0,-3);
    \coordinate (YAxisMax) at (0,0);
\draw [thin, black] (0,0) -- (4,0);
    \draw [thin, black] (0,-1) -- (4,-1);
    \draw [thin, black] (0,-2) -- (3,-2);
    \draw [thin, black] (0,-3) -- (3,-3);
    \draw [thin, black] (0,-4) -- (1,-4);
    \draw [thin, black] (0,0) -- (0,-4);
    \draw [thin, black] (1,0) -- (1,-4);
    \draw [thin, black] (2,0) -- (2,-3);
    \draw [thin, black] (3,0) -- (3,-3);
    \draw [thin, black] (4,0) -- (4,-1);
    \filldraw[fill = gray]
    (3,0) rectangle (4,-1) (2,-1)rectangle(3,-2) (2,-2)rectangle(3,-3) (1,-2)rectangle(2,-3) (0,-3)rectangle(1,-4);
    \end{tikzpicture}
\end{gather*}
If a GBS $\theta=\la/\mu$ has $m$ connected components $(\xi_1,\xi_2,\ldots,\xi_m)$, we define the weight of $\theta$ by
$$wt(\theta;t)= (t-1)^{m-1}\prod\limits_{i=1}^{m}(-1)^{r(\xi_{i})-1}t^{c(\xi_i)-1},$$
where the $r(\xi_i)$ (resp. $c(\xi_i)$) denotes the number of rows (resp. columns) in the border strip $\xi_i$. We make the convention that $wt(\la/\mu;t)=1$ if $|\la/\mu|=0$.

More generally, we define
\begin{align}\label{e:defwt2}
wt(\la/\mu;k,t):=
\begin{cases}
t^{k-1}wt(\la/\mu;t), & |\la/\mu|=0;\\
(t-1)t^{k-|\la/\mu|-1}wt(\la/\mu;t), & 0<|\la/\mu|<k;\\
wt(\la/\mu;t), & |\la/\mu|=k;\\
0, & |\la/\mu|>k.
\end{cases}
\end{align}

Due to \cite[Corollary 4.2]{JL}, we have
\begin{align}\label{e:q*S}
q^*_{k}(t)S_{\la}.1=\sum_{\mu}t^{k-1}(1-t)wt(\la/\mu;t^{-1})S_{\mu}.1 \quad (k>0)
\end{align}
summed over all partitions $\mu$ such that $\la/\mu$ is a $k$-GBS.

We rewrite \eqref{e:relq^} as
\begin{align}\label{e:relq^2}
\frac{t^k}{t-1}\hat{q}_k(t^{-1})=\frac{t^k}{t-1}q_{k}(t^{-1})+\sum_{i=1}^{k-1}t^{k-1}q_{i}(t^{-1})+t^{k-1}.
\end{align}

Now we give an alternative proof of the following Murnaghan-Nakayama rule obtained first in \cite[Prop. 3.3]{DHP}.

\begin{thm}
Let $k$ ba a positive integer, then
\begin{align}
\frac{t^k}{t-1}\hat{q}^*_k(t^{-1})S_{\la}.1=\sum_{\mu}wt(\la/\mu;k,t)S_{\mu}.1
\end{align}
summed over partitions $\mu$ such that $\la/\mu$ is a GBS with $0\leq|\la/\mu|\leq k$. In terms of irreducible characters, the iterative Murnaghan-Nakayama rule takes the following form
\begin{align}\label{e:M-N}
\chi^{\la}_{\mu}(q)=\sum_{\nu}wt(\la/\nu;\mu_1,q)\chi^{\nu}_{\mu^{[1]}}(q)
\end{align}
summed over partitions $\nu$ such that $|\nu|\leq|\mu^{[1]}|$ and $\la/\nu$ is a GBS with $0\leq|\la/\nu|\leq \mu_1$.
\end{thm}
\begin{proof}
It follows from immediate calculation that
\begin{align*}
&\frac{t^k}{t-1}\hat{q}^*_k(t^{-1})S_{\la}.1\\
=&\frac{t^k}{t-1}q^*_{k}(t^{-1})S_{\la}.1+\sum_{i=1}^{k-1}t^{k-1}q^*_{i}(t^{-1})S_{\la}.1+t^{k-1}S_{\la}.1 \quad \text{(by \eqref{e:relq^2})}\\
=&\sum_{\mu\vdash |\la|-k}wt(\la/\mu;t)S_{\mu}.1+\sum_{i=1}^{k-1}t^{k-1}\sum_{\mu\vdash |\la|-i}(t-1)t^{-i}wt(\la/\mu;t)S_{\mu}.1+t^{k-1}S_{\la}.1 \quad \text{(by \eqref{e:q*S})}\\
=&\sum_{\mu}wt(\la/\mu;k,t)S_{\mu}.1 \quad \text{(by \eqref{e:defwt2})}
\end{align*}
summed over partitions $\mu$ such that $\la/\mu$ is a GBS with $0\leq|\la/\mu|\leq k$.
\end{proof}

\subsection{Compact formulae for hook and two-row types}
In this subsection, we will consider character formulae for special cases based on the new iterative formula derived in the previous subsection.  We first consider the following simple cases.
\begin{prop}
Let $\la\vdash m$, $\mu\vdash n$ and $m\leq n$,
\begin{align}\label{e:first}
\chi^{\la}_{(1^n)}(q)&=\binom{n}{m}\frac{m!}{\prod_{(i,j)\in\la} h_{ij}},\\\label{e:second}
\chi^{(m)}_{\mu}(q)&=\frac{q^n}{(q-1)^{l(\mu)}}\sum_{\tau\in C(\mu;m)}(1-q^{-1})^{l(\tau)+l(\mu-\tau)},\\\label{e:third}
\chi^{(1^m)}_{\mu}(q)&=\frac{q^n}{(q-1)^{l(\mu)}}\sum_{\tau\in C(\mu;n-m)}(1-q^{-1})^{l(\tau)+l(\mu-\tau)}(-q^{-1})^{m-l(\mu-\tau)},\\\label{e:fouth}
\chi^{\la}_{(n)}(q)&=\sum_{\nu}wt(\la/\nu;n,q)
\end{align}
summed over partitions $\nu$ such that $|\nu|\leq|\mu^{[1]}|$ and $\la/\nu$ is a GBS with $0\leq|\la/\nu|\leq \mu_1$.
\end{prop}
\begin{proof}
\eqref{e:second} and \eqref{e:fouth} are immediate results of \eqref{e:ite} and \eqref{e:M-N} respectively. For \eqref{e:first}, we have $\chi^{\la}_{(1^n)}(q)=\langle \frac{q^n}{(q-1)^n}\hat{q}_{(1^n)}(q^{-1}), s_{\la} \rangle=\langle (p_1+1)^n, s_{\la} \rangle=\binom{n}{m}\langle p_1^m, s_{\la} \rangle=\binom{n}{m}\frac{m!}{\prod h_{ij}}$, in which the last identity follows from the well-known hook formula of the symmetric group $\mathfrak{S}_n$ \cite{Sag}. For \eqref{e:third}, one has
\begin{align*}
\chi^{(1^m)}_{\mu}(q)=&\left\langle \frac{q^n}{(q-1)^{l(\mu)}}\hat{q}_{\mu}(q^{-1}), s_{(1^m)} \right\rangle\\
=&\frac{q^n}{(q-1)^{l(\mu)}}\left\langle \sum_{\tau\subset\mu}(1-q^{-1})^{l(\tau)}q_{\mu-\tau}(q^{-1}), s_{(1^m)} \right\rangle \quad\text{(by \eqref{e:relq^})}\\
=&\frac{q^n}{(q-1)^{l(\mu)}}\sum_{\tau\in C(\mu;n-m)}(1-q^{-1})^{l(\tau)+l(\mu-\tau)}(-q^{-1})^{m-l(\mu-\tau)}\quad\text{(by \cite[(2.27)]{JL})}.
\end{align*}
\end{proof}

For $\mu\vdash n$, we define two families of polynomials in $t$ associated with $\mu$ by
\begin{align*}
    a_{ij}(\mu;t)&=\sum_{\tau\in C(\mu;i)}\sum_{\theta\in C(\mu-\tau;j)}(1-t^{-1})^{l(\tau)+l(\theta)}(1-t)^{l(\mu-\tau-\theta)};\\
    b_{ij}(\mu;t)&=\sum_{\tau\in C(\mu;i)}\sum_{\theta\in C(\mu-\tau;j)}(1-t^{-1})^{l(\tau)+l(\theta)+l(\mu-\tau-\theta)},
\end{align*}
for $i+j\leq n$ and $a_{ij}(\mu;t)=b_{ij}(\mu;t)=0$ for $i+j>n$.

Symmetry of $\tau$ and $\theta$ yields
\begin{align} \label{a}
    a_{ij}(\mu;t)&=a_{ji}(\mu;t),\\ \label{b}
    b_{ij}(\mu;t)&=b_{ji}(\mu;t).
\end{align}
It is clear that $a_{0,0}(\mu;t)=(1-t)^{l(\mu)}$ and $b_{0,0}(\mu;t)=(1-t^{-1})^{l(\mu)}$.

\begin{prop}\label{t:generating}
We have the generating functions of $a_{ij}$ and $b_{ij}$ as follows:
\begin{align}
\sum_{i=0}^\infty\sum_{j=0}^\infty a_{ij}(\mu;t)v^iu^j&=\prod_{i=1}^{l(\mu)}\left(\sum_{\tau_i=0}^{\mu_i}\sum_{\theta_i=0}^{\mu_i-\tau_i}(1-t^{-1})^{1-\delta_{\tau_i,0}}(1-t^{-1})^{1-\delta_{\theta_i,0}}
(1-t)^{1-\delta_{\mu_i,\tau_i+\theta_i}}v^{\tau_i}u^{\theta_i}\right),\\
\sum_{i=0}^\infty\sum_{j=0}^\infty b_{ij}(\mu;t)v^iu^j&=\prod_{i=1}^{l(\mu)}\left(\sum_{\tau_i=0}^{\mu_i}\sum_{\theta_i=0}^{\mu_i-\tau_i}(1-t^{-1})^{1-\delta_{\tau_i,0}}(1-t^{-1})^{1-\delta_{\theta_i,0}}(1-t^{-1})^{1-\delta_{\mu_i,\tau_i+\theta_i}}v^{\tau_i}u^{\theta_j}\right).
\end{align}
\end{prop}
\begin{proof}
It follows from immediate calculations that
\begin{align*}
&\sum_{i=0}^\infty\sum_{j=0}^\infty a_{ij}(\mu;t)v^iu^j\\
=&\sum_{i=0}^\infty\sum_{j=0}^\infty\sum_{\tau\in C(\mu;i)}\sum_{\theta\in C(\mu-\tau;j)}(1-t^{-1})^{l(\tau)+l(\theta)}(1-t)^{l(\mu-\tau-\theta)}v^{i}u^j\\
=&\sum_{i=0}^n\sum_{j=0}^n\sum\limits_{\tau_1+\ldots+\tau_l=i}\sum\limits_{\theta_1+\ldots+\theta_l=j}(1-t^{-1})^{2l-\sum_{s=1}^l\delta_{\tau_s,0}-\sum_{s=1}^l\delta_{\theta_s,0}}
(1-t)^{l-\sum_{s=1}^l\delta_{\mu_s,\tau_s+\theta_s}}v^{i}u^j\\
=&\prod_{i=1}^l\left(\sum_{\tau_i=0}^{\mu_i}\sum_{\theta_i=0}^{\mu_i-\tau_i}(1-t^{-1})^{1-\delta_{\tau_i,0}}(1-t^{-1})^{1-\delta_{\theta_i,0}}
(1-t)^{1-\delta_{\mu_i,\tau_i+\theta_i}}v^{\tau_i}u^{\theta_i}\right).
\end{align*}
The second identity can be derived similarly.
\end{proof}

Prop. \ref{t:generating} gives the following special cases, which is essential to the subsequent contents.

\begin{cor}\label{t:identity}
For a partition $\mu=(\mu_1,\ldots,\mu_l)\vdash n$, we have
\begin{align}
    &\sum_{i=0}^\infty\sum_{j=0}^\infty a_{ij}(\mu;q)q^{i+j}=\prod_{i=1}^l(q-1)q^{\mu_i-1},\\\label{a,q}
    &\sum_{i=0}^\infty\sum_{j=0}^\infty a_{ij}(\mu;q)(-q)^iq^j=\prod_{i=1}^l(1-q)(-q)^{\mu_i-1},\\\label{a,-q}
    &\sum_{i=0}^\infty\sum_{j=0}^\infty a_{ij}(\mu;q)(-q)^i(-q)^j=\prod_{i=1}^l(1-q)A(\mu_i;q),\\\label{b,1}
        &\sum_{i=0}^\infty\sum_{j=0}^\infty b_{ij}(\mu;q)=(1-q^{-1})^l\prod_{i=1}^l\left( 3\mu_i-3q^{-1}(\mu_i-1)+\dfrac{(\mu_i-1)(\mu_i-2)(1-q^{-1})^2}{2}\right),
\end{align}
where $A(\mu_i;q)=\left(\frac{(-q)^{\mu_i-1}(q^2+6q+1)+4}{(q+1)^2}+\frac{2(-q)^{\mu_i-1}(q-1)\mu_i}{q+1}\right)$.
\end{cor}

Now we present our main results: compact formulae for $\chi^{\la}_{\mu}(q)$ with $\la$ hook and two-row type.

\begin{thm}\label{t:compact}
Let $\mu\vdash n$ and $1\leq k\leq m\leq n$, we have
\begin{align}\label{e:hook}
\chi_{\mu}^{(k,1^{m-k})}(q)&=\frac{q^{n-m}}{(q-1)^{l(\mu)}}\sum_{i=k}^m(-1)^{m-k}a_{i,n-m}(\mu;q)q^i,\\\label{e:two}
\chi_{\mu}^{(k,m-k)}(q)&=\frac{q^n}{(q-1)^{l(\mu)}}\left(b_{k,m-k}(\mu;q)-b_{k+1,m-k-1}(\mu;q)\right).
\end{align}
\begin{proof}
For the first identity, by Theorem \ref{X,X}, all possible $\nu$'s can be chosen from $\{(1^j)\mid 0\leq j\leq m-k\}$, then
\begin{align*}
    \chi^{(k,1^{m-k})}_\mu(q)&=\sum_{j=0}^{m-k}\sum_{\tau\in C(\mu;m-j)}(-1)^{m-k-j}\dfrac{q^{m-j-l(\tau)}}{(q-1)^{l(\mu)-l(\tau)-l(\mu-\tau)}}\chi_{\mu-\tau}^{(1^j)}(q) \\
    &=\sum_{j=0}^{m-k}\sum_{\tau\in C(\mu;m-j)}(-1)^{m-k-j}\dfrac{q^{m-j-l(\tau)}}{(q-1)^{l(\mu)-l(\tau)-l(\mu-\tau)}}
    \dfrac{q^{n-m+j}}{(q-1)^{l(\mu-\tau)}}\\
    &\sum_{\theta\in C(\mu-\tau;n-m)}(1-q^{-1})^{l(\theta)+l(\mu-\tau-\theta)}(-q^{-1})^{j-l(\mu-\tau-\theta)}\quad\text{(by \eqref{e:third})}.
\end{align*}
Replacing $j\rightarrow m-i$ and simplification give \eqref{e:hook}. The second identity is obtained similarly. Just note $\nu=(m-k)$ or $(m-k-1)$ in this case.

\end{proof}
\end{thm}

\begin{exmp} The following two examples show effectiveness of our compact formulas.

 (i) Suppose $\la=(31^2)$, $\mu=(321)$, we have
    \begin{align}\label{e:E1}
    \chi_{(321)}^{(31^2)}(q)&=\frac{q}{(q-1)^3}\left(a_{31}(\mu;q)q^3+a_{41}(\mu;q)q^4+a_{51}(\mu;q)q^5\right).
    \end{align}
   Theorem \ref{t:generating} gives
    \begin{align*}
    \sum_{i=0}^6\sum_{i=0}^6a_{ij}v^iu^j=&(1-t^{-1})^3\left[-t+(v+u)\right]\left[-t+(v^2+u^2)+(1-t)(u+v)+(1-t^{-1})vu\right]\\
    &\left[-t+(v^3+u^3)+(1-t)(v+v^2+u+u^2)+(1-t^{-1})(v^2u+vu^2)+(1-t^{-1})(1-t)uv\right].
\end{align*}
By taking the coefficient of $v^3u$, $v^4u$, $v^5u$, we conclude that
    $$\begin{cases}a_{31}(\mu;q)=(1-q^{-1})^3(7q^2-19q+13-2q^{-1})\\a_{41}(\mu;q)=(1-q^{-1})^3(11-8q-3q^{-1})\\a_{51}(\mu;q)=2(1-q^{-1})^4.\end{cases}$$
    Substituting them into \eqref{e:E1} yields $\chi_{(321)}^{(31^2)}(q)=q^3-10q^2+10q-2$.

(ii) Suppose $\la=(32)$, $\mu=(2^3)$, we have
    \begin{align}\label{e:E2}
    \chi_{(2^3)}^{(32)}(q)=\frac{q^6}{(q-1)^3}\left(b_{32}(\mu;q)-b_{41}(\mu;q)\right).
    \end{align}
    Theorem \ref{t:generating} implies
    \begin{align*}
        \sum_{i=0}^6\sum_{j=0}^6b_{ij}v^iu^j=(1-t^{-1})^3\left[v^2+u^2+(1-t^{-1})(vu+v+u)+1\right]^3.
    \end{align*}
By taking the coefficients of $v^3u^2$, $v^4u$, we conclude
$$\begin{cases}
    b_{32}(\mu;q)=6(1-q^{-1})^5+6(1-q^{-1})^4\\
    b_{41}(\mu;q)=6(1-q^{-1})^5+3(1-q^{-1})^4.
\end{cases}$$
Plugging them into \eqref{e:E2} gives $\chi_{(2^3)}^{(32)}(q)=3q^2(q-1)$.
\end{exmp}

Now we study the sign $q$-permutation representation of the rook monoid
$R_n(q)$. Let $V=V_{\bar{0}}\bigoplus V_{\bar{1}}$ be the $\mathbb{Z}_2$-graded vector space over $\mathbb{C}(q)$ with basis $\{v_1,v_2,\cdots,v_{a+b}\}$,
 where $V_{\bar 0}=\langle v_{1},v_{2},\cdots,v_{a}\rangle$ and $V_{\bar 1}=\langle v_{a+1},v_{a+2},\cdots,v_{a+b}\rangle$

Define the endomorphism $\pi$ of $V\otimes V$ by 
\begin{align}
\pi(v_k\otimes v_l)=
\begin{cases}
(-1)^{|v_k||v_l|}v_l\otimes v_k+(q-1)v_k\otimes v_l, k<l\\
\frac{(-1)^{|v_k|}(q+1)+q-1}{2}v_k\otimes v_l, k=l\\
(-1)^{|v_k||v_l|}qv_l\otimes v_k, k>l,
\end{cases}
\end{align}
and let $\pi_i$ be the endomorphism of $V^{\otimes n}$ by letting $\pi$ acting on the $(i, i+1)$ factor and the trivial identity action
elsewhere, $i=1, \cdots, n-1$. The action $\pi$ gives rise to the so-called $q$-permutation representation $\Phi^{q,n}_{a,b}$ of $H_n(q)$.
In particular, $\varphi^{n}_{1,1}$ and $\varphi^n_{2, 0}$ decomposes themselves \cite{Zhao} into some special direct sums of irreducible representations
of $H_n(q)$:
\begin{align}\label{e:1,1}
\varphi^{n}_{1,1}&=\sum\limits_{k =0}^{n-1}\zeta^{(n-k,1^k)},\\\label{e:2,0}
\varphi^{n}_{2,0}&=\sum\limits_{k=0}^{\lceil n/2\rceil}(n-2k+1)\zeta^{(n-k,k)},
\end{align}
where $\zeta^{\la}$ is the irreducible character of the Hecke algebra indexed by $\la$ and $\lceil n/2\rceil$ denotes the maximum integer number $\leq\frac{n}{2}$. In \cite{JL} two compact formulas were given for $\varphi^{n}_{1,1}, \varphi^{n}_{2,0}$:
\begin{align}\label{e:hooksum}
\varphi^n_{1,1}(T_{\gamma_{\mu}})&=\sum\limits_{i=0}^{n-1}\zeta^{(n-i,1^i)}_{\mu}(q)
=(-1)^{n-l(\mu)}2^{l(\mu)-1}\prod\limits_{i=1}^{l(\mu)}[\mu_i]_{-q},
\\\label{e:twosum}
\varphi^n_{2,0}(T_{\gamma_{\mu}})&=\sum\limits_{i=0}^{\lceil n/2\rceil}(n-2i+1)\zeta^{(n-i,i)}_{\mu}(q)
=q^{n-2l(\mu)}\prod\limits_{i=1}^{l(\mu)}(1+q+\mu_i(q-1)),
\end{align}
where $[n]_q$ is the $q$-binomial number defined by $[n]_q:=\frac{q^n-1}{q-1}$.

The following results seem analogs of \eqref{e:hooksum} and \eqref{e:twosum} for the q-rook monoid $R_q(n)$.
\begin{thm}\label{t:perm}
Under the same assumption of Theorem \ref{t:compact}, one has
\begin{align}\label{e:fsum}
\sum_{m=0}^n\sum_{k=0}^{m-1}\chi_\mu^{(m-k,1^k)}(q)&=\frac{(-1)^{n+l(\mu)}}{2}\left(\prod_{i=1}^{l(\mu)}A(\mu_i;q)-(-q)^{n-l(\mu)}\right),\\\label{e:ssum}
\sum_{m=0}^n\sum_{k=0}^{\lceil m/2\rceil}(m-2k+1)\chi_\mu^{(m-k,k)}(q)
&=q^{n-l(\mu)}\prod_{i=1}^{l(\mu)}\left( 3\mu_i-3q^{-1}(\mu_i-1)+\dfrac{(\mu_i-1)(\mu_i-2)(1-q^{-1})^2}{2}\right).
\end{align}
\end{thm}
\begin{proof}
For the first identity, by \eqref{e:hook}, we have
\begin{align}\label{e:I1}
\sum_{m=0}^n\sum_{k=0}^{m-1}\chi_\mu^{(m-k,1^k)}(q)&=\sum_{m=0}^n\sum_{k=0}^{m-1}\frac{(-1)^k}{(q-1)^l}\sum_{i=m-k}^ma_{i,n-m}(\mu;q)q^iq^{n-m}.
\end{align}
Note that
\begin{align*}
&\sum_{k=0}^{m-1}\sum_{i=m-k}^m (-1)^ka_{i,n-m}(\mu;q)q^i\\
=&\sum_{i=1}^{m}\sum_{k=m-i}^{m-1} (-1)^ka_{i,n-m}(\mu;q)q^i \quad\text{(exchange the sums)}\\
=&\sum_{i=1}^{m}\frac{(-1)^m((-1)^i-1)}{2}a_{i,n-m}(\mu;q)q^i.
\end{align*}
Since $\frac{(-1)^m((-1)^i-1)}{2}=0$ when $i=0$ and $a_{i,j}(\mu;q)=0$ when $i+j>|\mu|$,
\begin{align*}
\sum_{k=0}^{m-1}\sum_{i=m-k}^m (-1)^ka_{i,n-m}(\mu;q)q^i
=\sum_{i=0}^{n}\frac{(-1)^m((-1)^i-1)}{2}a_{i,n-m}(\mu;q)q^i.
\end{align*}
Substituting this to \eqref{e:I1} gives
\begin{align*}
&\sum_{m=0}^n\sum_{k=0}^{m-1}\chi_\mu^{(m-k,1^k)}(q)\\
=&\frac{(-1)^n}{(q-1)^{l(\mu)}}\sum_{m=0}^n\sum_{i=0}^n\frac{(-1)^i-1}{2}a_{i,n-m}(\mu;q)q^i{(-q)}^{n-m}\\
=&\frac{(-1)^n}{(q-1)^{l(\mu)}}\sum_{m=0}^n\sum_{i=0}^n\frac{(-1)^i-1}{2}a_{i,m}(\mu;q)q^i{(-q)}^{m}\quad\text{(replace $m$ by $n-m$)}\\
=&\frac{(-1)^n}{(q-1)^{l(\mu)}}\sum_{m=0}^{\infty}\sum_{i=0}^{\infty}\frac{(-1)^i-1}{2}a_{i,m}(\mu;q)q^i{(-q)}^{m}\quad\text{($a_{i,j}(\mu;q)=0$ if $i+j>n$)}\\
=&\frac{(-1)^n}{2(q-1)^{l(\mu)}}\left(\sum_{i=0}^{\infty}\sum_{m=0}^{\infty}a_{i,m}(\mu;q)(-q)^i(-q)^m-\sum_{m=0}^{\infty}\sum_{i=0}^{\infty}a_{m,i}(\mu;q)(-q)^mq^i \right)\quad\text{($a_{i,j}(\mu;q)=a_{j,i}(\mu;q)$)}.
\end{align*}
Then \eqref{e:fsum} follows from Corollary \ref{t:identity}.

For the second identity, we have
\begin{align*}
&\sum_{m=0}^n\sum_{k=0}^{\lceil m/2\rceil}(m-2k+1)\chi_\mu^{(m-k,k)}(q)\\
=&\sum_{m=0}^n\frac{q^n}{(q-1)^{l(\mu)}}\sum_{k=0}^{\lceil m/2\rceil}(m-2k+1)(b_{m-k,k}(\mu;q)-b_{m-k+1,k-1}(\mu;q))\quad\text{(by \eqref{e:two})}\\
=&\frac{q^n}{(q-1)^{l(\mu)}}\sum_{m=0}^n\sum_{k=0}^mb_{m-k,k}(\mu;q)\quad\text{($b_{i,j}(\mu;q)=b_{j,i}(\mu;q)$)}\\
=&\frac{q^n}{(q-1)^{l(\mu)}}\sum_{m=0}^{\infty}\sum_{k=0}^{\infty}b_{m,k}(\mu;q)\quad\text{($b_{i,j}(\mu;q)=0$ if $i+j>n$)}.
\end{align*}
Then Corollary \ref{t:identity} gives \eqref{e:ssum}.
\end{proof}

It would be interesting to construct the representation of $H_q(n)$ afforded by Theorem \ref{t:perm}.

\section{The second orthogonal relation}\label{s:bitrace}

In this section, we introduce the bitrace of $R_n(q)$ as an analogue of the bitrace of the Iwahori-Hecke
algebra $H_n(q)$ in type $A$ \cite{HLR}.

Recall that $\{T_{(A,B,\omega)}|(A,B,\omega\in\Omega)\}$ forms a linear $\mathbb{C}(q)$-basis of $R_n(q)$. Define the bitrace of elements $T_{(C,D,x)}$ and $T_{(E,F,y)}$ by
\begin{align}
 btr(T_{(C,D,x)},T_{(E,F,y)}):=\sum_{(A,B,\omega)\in\Omega}T_{(C,D,x)}T_{(A,B,\omega)}T_{(E,F,y)}|_{T_{(A,B,\omega)}}
\end{align}
where $T_{(C,D,x)}T_{(A,B,\omega)}T_{(E,F,y)}|_{T_{(A,B,\omega)}}$ denotes the coefficient of the basis element $T_{(A,B,\omega)}$ in the linear expansion of $T_{(C,D,x)}T_{(A,B,\omega)}T_{(E,F,y)}$. If $(A,B,\omega)\in\Omega$, let $L_{(A,B,\omega)}$ (resp. $R_{(A,B,\omega)}$) denote the linear transformation of $R_n(q)$ induced by the action of $(A,B,\omega)$ on $R_n(q)$ by the left (resp. right) multiplication, then $L_{(A,B,\omega)}$ and $R_{(A,B,\omega)}$ are commutative each other and
\begin{align}
btr(T_{(C,D,x)},T_{(E,F,y)})=Tr(L_{(C,D,x)}R_{(E,F,y)}).
\end{align}

The algebra $R_n(q)$ becomes a $R_n(q)$-bimodule under the left and right multiplication.
Since $R_n(q)$ is semisimple, by double centralizer theory, we have that
\begin{align}
    R_n(q)\cong\bigoplus_{k=0}^n\bigoplus_{\la\vdash k}R_\lambda\otimes R^{\la}
\end{align}
where $R_\lambda$ is the irreducible left $R_n(q)$-module labeled by $\la$, $R^{\la}$ is the irreducible right $R_n(q)$-module labeled by $\la$, then
\begin{align}
    btr(T_{(C,D,x)},T_{(E,F,y)})=\sum_{k=0}^n\sum_{\la\vdash k}\chi_{R_n(q)}^\la(T_{(C,D,x)})\chi_{R_n(q)}^\la(T_{(E,F,y)}).
\end{align}

Note that any irreducible character of $R_n(q)$ is completely determined by its values on the standard elements $T_{\mu},\mu\vdash k,k\leq n$ \cite[Theorem 5.5]{DHP}. For this reason, we can define the bitrace of arbitrary two compositions $\mu, \nu\models n$ as
\begin{align}
    btr(\mu,\nu):=btr(T_{\mu},T_{\nu})=\sum_{k=0}^n\sum_{\lambda\vdash k}\chi^\la_\mu(q)\chi^\la_\nu(q).
\end{align}

Using \eqref{chara formula} and \eqref{e:relq^} respectively, one has
\begin{align}\label{e:btr1}
\begin{split}
    btr(\mu,\nu)=&\dfrac{q^{2n}}{(q-1)^{l(\mu)+l(\nu)}}\langle\hat{q}_\mu(x;q^{-1}),\hat{q}_\nu(x;q^{-1}) \rangle\\
    =&\dfrac{q^{2n}}{(q-1)^{l(\mu)+l(\nu)}}\left\langle \sum_{\tau\subset\mu}(1-q^{-1})^{l(\mu-\tau)}q_{\tau}(q^{-1}),\sum_{\theta\subset\nu}(1-q^{-1})^{l(\nu-\theta)}q_\theta(q^{-1})\right\rangle \\ =&\dfrac{q^{2n}}{(q-1)^{l(\mu)+l(\nu)}}\sum_{k=0}^{n}\sum_{\substack{\tau\in C(\mu;k)\\\theta\in C(\nu;k)}}(1-q^{-1})^{l(\mu-\tau)+l(\nu-\theta)}\langle q_\tau(q^{-1}),q_\theta(q^{-1})\rangle\\
    =&\dfrac{q^{2n}}{(q-1)^{l(\mu)+l(\nu)}}\sum_{k=0}^{n}\sum_{\substack{\tau\in C(\mu;k)\\\theta\in C(\nu;k)}}(1-q^{-1})^{l(\tau)+l(\theta)}\langle q_{\mu-\tau}(q^{-1}),q_{\nu-\theta}(q^{-1})\rangle.
    \end{split}
\end{align}

For simplicity we denote for integer $r>0$
\begin{align*}
(r)_{t}=(t-1)\frac{t^{2r}-1}{t+1}
\end{align*}
and make the convention that $(r)_t = 0$ for $r<0$ and $(0)_t = 1$.

By \cite[(5.8)]{JL} we have that
\begin{align}\label{e:<q,q>}
\langle q_{\mu-\tau}(q^{-1}),q_{\nu-\theta}(q^{-1})\rangle=\frac{1}{q^{2n-2k}}\sum\limits_{M}\prod_{m_{ij}}(m_{ij})_q.
\end{align}
summed over all $l(\mu)\times l(\nu)$ nonnegative integer matrices $M=(m_{ij})$ such that
the row sums are $\mu_1-\tau_1,\ldots,\mu_{l(\mu)}-\tau_{l(\mu)}$ and the column sums are $\nu_1-\theta_1,\ldots,\nu_{l(\nu)}-\theta_{l(\nu)}$.

Substituting \eqref{e:<q,q>} into \eqref{e:btr1} gives that
\begin{align}\label{e:btr2}
\begin{split}
&btr(\mu,\nu)\\
=&\dfrac{1}{(q-1)^{l(\mu)+l(\nu)}}\sum_{k=0}^{n}\sum_{\substack{\tau\in C(\mu;k)\\\theta\in C(\nu;k)}}(1-q^{-1})^{l(\tau)+l(\theta)}q^{2k}\sum\limits_{M}\prod_{m_{ij}}(m_{ij})_q\\
=&\dfrac{1}{(q-1)^{l(\mu)+l(\nu)}}\sum_{k=0}^{n}\sum_{\substack{\tau\in C(\mu;k)\\\theta\in C(\nu;k)}}\sum\limits_{M}\prod_{i\geq 1}^{l(\mu)}(1-q^{-1})^{1-\delta_{0,\tau_i}}q^{\tau_i}\prod_{j\geq 1}^{l(\nu)}(1-q^{-1})^{1-\delta_{0,\theta_{i}}}q^{\theta_i}\prod_{m_{ij}}(m_{ij})_q
\end{split}
\end{align}
summed over all $l(\mu)\times l(\nu)$ nonnegative integer matrices $M=(m_{ij})$ such that
the row sums are $\mu_1-\tau_1,\ldots,\mu_{l(\mu)}-\tau_{l(\mu)}$ and the column sums are $\nu_1-\theta_1,\ldots,\nu_{l(\nu)}-\theta_{l(\nu)}$.

For an arbitrary nonnegative integer matrix $M=(m_{ij})$, define the weight of element $m_{ij}$ by
\begin{align}
wt(m_{ij}):=
\begin{cases}
(1-q^{-1})^{1-\delta_{0,m_{ij}}}q^{m_{ij}}, & \text{$i=1$ or $j=1$};\\
(m_{ij})_q, & \text{$i>1$ and $j>1$}.
\end{cases}
\end{align}
Moreover, the weight of $M$ is defined as $wt(M):=\prod_{m_{ij}}wt(m_{ij})$.

Then we can derive the following combinatorial formula for $btr(\mu,\nu)$.
\begin{thm}\label{t:bitrace}
Let $\mu,\nu\models n$, then
\begin{align}
btr(\mu,\nu)=\dfrac{1}{(q-1)^{l(\mu)+l(\nu)}}\sum_{M}wt(M)
\end{align}
summed over all $(l(\mu)+1)\times (l(\nu)+1)$ nonnegative integer matrices $M=(m_{ij})$ satisfying \\
(i) $m_{11}=0$;\\
(ii) the first row sum equals to the first column sum;\\
(iii) the $i$-th row sum is $\mu_i$ and the $j$-th column sum is $\nu_j$, $2\leq i\leq l(\mu)$, $2\leq j\leq l(\nu)$.
\end{thm}
\begin{cor}
    Let $T_\mu\in R_n(q)$ and $\mu\models n$, then the trace of regular representation of $R_n(q)$ at element $T_\mu$ is given by
    \begin{align}\label{e:reg}
         \chi^{reg}(T_\mu)=\frac{q^n}{(q-1)^{l(\mu)}}\sum_{i=0}^k\sum_{\tau\in C(\mu;i)}\binom{n}{i}(1-q^{-1})^{l(\mu-\tau)+i}\dfrac{i!}{\prod_j\tau_j!}.
    \end{align}
\end{cor}
\begin{proof}
    By definition, $\chi^{reg}(T_\mu)=btr(\mu,(1^n))$. Let us first fix the first row (column) sum in $M$ to be $i$ ($0\leq i\leq n$), then the first row in $M$ has $\binom{n}{i}$ possibilities. If $\tilde{M}$ denotes the submatrix of $M$ obtained by removing the first row and the first column, then the entries in $\tilde{M}$ are $0$ or $1$. Assume that we choose the first column to be $(0,\tau_1,\cdots,\tau_{l(\mu)})\subset\mu$, then each column has at most one $1$ and the $j$-th row has $\mu_j-\tau_j$ $1$'s in $\tilde{M}$. Thus the number of these $\tilde{M}$ is $\dfrac{(n-i)!}{\prod_j(\mu_j-\tau_j)!}$. In this case, $wt(M)=(1-q^{-1})^{l(\tau)+i}q^{2i}(q-1)^{2n-2i}$. Hence $btr(\mu,(1^n))=\frac{1}{(q-1)^{l(\mu)+n}}\sum_{i=0}^{n}\sum_{\tau\in C(\mu;i)}\binom{n}{i}\dfrac{(n-i)!}{\prod_j(\mu_j-\tau_j)!}(1-q^{-1})^{l(\tau)+i}q^{2i}(q-1)^{2n-2i}.$ Substituting $i\rightarrow n-i$ and $\tau\rightarrow \mu-\tau$ gives \eqref{e:reg}.
\end{proof}

We remark that \eqref{e:reg} reduces to $\chi^{reg}(1)=\sum_{i=0}^n\binom{n}{i}^2i!$ when $\mu=(1^n)$, which gives the dimension of $R_{n}(q)$.

\appendix
\section{A table of $\chi^{\la}_{\mu}(q)$}\label{app}

In \cite{DHP}, a table of $\chi^{\la}_{\mu}(q)$ $(|\mu|\leq4)$ was presented. Here we list $\chi^{\la}_{\mu}(q)$ with $|\mu|=5$. Since $\chi^{\la}_{\mu}(q)$ equals to the irreducible character of the Hecke algebra when $|\la|=|\mu|$ (see \cite{Ram}), we just list those with $|\mu|=5$ and $|\la|<5$.
\begin{table}[H]

 \centering

\caption{\label{tab:3}}

\begin{tabular}{|c|c|c|c|c|c|c|c|}

  \hline

 \tabincell{c}{$\lambda\backslash \mu$ } & $(1^5)$ & $(21^3)$ &$(2^21)$ & $(31^2)$ & $(32)$ & $(41)$ & $(5)$\\
%
%
%
%
%
%
%
%
%
%
%
%
%
%
%
\hline

$(1^4)$ & \tabincell{c}{$5$} & \tabincell{c}{$q-4$} & \tabincell{c}{$-2q+3$} & \tabincell{c}{$-q+3$} & \tabincell{c}{$2(q-1)$} & \tabincell{c}{$q-2$} & \tabincell{c}{$1-q$}\\

\hline

$(21^2)$ & \tabincell{c}{$15$} & \tabincell{c}{$3(2q-3)$} & \tabincell{c}{$2q^2-8q+5$} & \tabincell{c}{$q^2-5q+4$} & \tabincell{c}{$-3q^2+5q-2$} & \tabincell{c}{$-q^2+3q-1$} & \tabincell{c}{$q^2-q$}\\

\hline

$(2^2)$ & \tabincell{c}{$10$} & \tabincell{c}{$5(q-1)$} & \tabincell{c}{$3q^2-4q+3$} & \tabincell{c}{$q^2-4q+1$} & \tabincell{c}{$q^3-2q^2+2q-1$} & \tabincell{c}{$-q^2+q$} & 0\\

\hline

$(31)$ & \tabincell{c}{$15$} & \tabincell{c}{$3(3q-2)$} & \tabincell{c}{$5q^2-8q+2$} & \tabincell{c}{$4q^2-5q+1$} & \tabincell{c}{$2q^3-5q^2+3q$} & \tabincell{c}{$q^3-3q^2+q$} & \tabincell{c}{$-q^3+q^2$}\\

\hline

$(4)$ & \tabincell{c}{$5$} & \tabincell{c}{$4q-1$} & \tabincell{c}{$3q^2-2q$} & \tabincell{c}{$3q^2-q$} & \tabincell{c}{$2q^3-2q^2$} & \tabincell{c}{$2q^3-q^2$} & \tabincell{c}{$q^4-q^3$}\\

\hline

$(1^3)$ & \tabincell{c}{$10$} & \tabincell{c}{$2(2q-3)$} & \tabincell{c}{$q^2-6q+3$} & \tabincell{c}{$q^2-3q+3$} & \tabincell{c}{$-2q^2+4q-1$} & \tabincell{c}{$-q^2+2q-1$} & \tabincell{c}{$q^2-q$}\\

\hline

$(21)$ & \tabincell{c}{$20$} & \tabincell{c}{$11q-9$} & \tabincell{c}{$6q^2-10q+4$} & \tabincell{c}{$4q^2-7q+2$} & \tabincell{c}{$2q^3-6q^2+4q-1$} & \tabincell{c}{$q^3-3q^2+2q$} & \tabincell{c}{$-q^3+q^2$}\\

\hline

$(3)$ & \tabincell{c}{$10$} & \tabincell{c}{$7q-3$} & \tabincell{c}{$5q^2-4q+1$} & \tabincell{c}{$4q^2-3q$} & \tabincell{c}{$3q^3-3q^2+q$} & \tabincell{c}{$2q^3-2q^2$} & \tabincell{c}{$q^4-q^3$}\\

\hline

$(1^2)$ & \tabincell{c}{$10$} & \tabincell{c}{$2(3q-2)$} & \tabincell{c}{$3q^2-6q+1$} & \tabincell{c}{$3q^2-3q+1$} & \tabincell{c}{$q^3-4q^2+2q$} & \tabincell{c}{$q^3-2q^2+q$} & \tabincell{c}{$-q^3+q^2$}\\

\hline

$(2)$ & \tabincell{c}{$10$} & \tabincell{c}{$7q-3$} & \tabincell{c}{$5q^2-4q+1$} & \tabincell{c}{$4q^2-3q$} & \tabincell{c}{$3q^3-3q^2+q$} & \tabincell{c}{$2q^3-2q^2$} & \tabincell{c}{$q^4-q^3$}\\

\hline

$(1)$ & \tabincell{c}{$5$} & \tabincell{c}{$4q-1$} & \tabincell{c}{$3q^2-2q$} & \tabincell{c}{$3q^2-q$} & \tabincell{c}{$2q^3-2q^2$} & \tabincell{c}{$2q^3-q^2$} & \tabincell{c}{$q^4-q^3$}\\

\hline

$\emptyset$ & \tabincell{c}{$1$} & \tabincell{c}{$q$} & \tabincell{c}{$q^2$} & \tabincell{c}{$q^2$} & \tabincell{c}{$q^3$} & \tabincell{c}{$q^3$} & \tabincell{c}{$q^4$}\\

\hline

 \end{tabular}

\end{table}

\vskip30pt \centerline{\bf Acknowledgments}
The paper is partially supported by Simons Foundation grant No. 523868 and NSFC grant No. 12171303. The third author is supported by China Scholarship Council
and he would like to thank the hospitality of the Faculty of Mathematics at the University of Vienna during the writing of this paper.
\bigskip


\bibliographystyle{plain}

\end{document}